\documentclass[a4paper,12pt,reqno]{amsart}
\usepackage{amssymb}
\usepackage{amsmath}
\usepackage{ifthen}
\usepackage{enumitem}
\usepackage{graphicx}
\usepackage[a4paper,margin=0.85in]{geometry}
\nonstopmode \numberwithin{equation}{section}
\usepackage{amssymb}
\usepackage{ifthen}

\usepackage{amsmath}
\usepackage[T1]{fontenc} 
\usepackage{comment}

\theoremstyle{plain}
\newtheorem{thm}{Theorem}
\numberwithin{thm}{section}
\newtheorem{cor}{Corollary}
\numberwithin{cor}{section}
\newtheorem{lem}{Lemma}
\numberwithin{lem}{section}
\newtheorem{prop}{Proposition}

\newtheorem{conj}{Conjecture}

\theoremstyle{definition}
\newtheorem{defn}{Definition}[section]

\newtheorem{prob}{Problem}
\newtheorem{rem}{Remark}[section]

\newcounter{minutes}
\divide\time by 60
\newcounter{hours}
\multiply\time by 60
\addtocounter{minutes}{-\time}

\newcounter {own}
\def\theown {\thesection       .\arabic{own}}

\newenvironment{pf}[1][]{%
	\vskip 3mm
	\noindent
	\ifthenelse{\equal{#1}{}}%
	{{\slshape Proof. }}%
	{{\slshape #1.} }%
}%
{\qed\bigskip}

\theoremstyle{plain}
\newtheorem{Thm}{Theorem}

\newtheorem{Lem}{Lemma}

\numberwithin{equation}{section}
\def\be{\begin{equation}}
	\def\ee{\end{equation}}

\newcommand{\bee}{\begin{enumerate}}
	\newcommand{\eee}{\end{enumerate}}

\newcommand{\blem}{\begin{lem}}
	\newcommand{\elem}{\end{lem}}
\newcommand{\bthm}{\begin{thm}}
	\newcommand{\ethm}{\end{thm}}
\newcommand{\bcor}{\begin{cor}}
	\newcommand{\ecor}{\end{cor}}
\newcommand{\beg}{\begin{examp}}
	\newcommand{\eeg}{\end{examp}}
\newcommand{\begs}{\begin{examples}}
	\newcommand{\eegs}{\end{examples}}
\allowdisplaybreaks
\newcommand{\bdefn}{\begin{defn}}
	\newcommand{\edefn}{\end{defn}}

\newcommand{\bprob}{\begin{prob}}
	\newcommand{\eprob}{\end{prob}}
\newcommand{\bei}{\begin{itemize}}
	\newcommand{\eei}{\end{itemize}}

\newcommand{\bcon}{\begin{conj}}
	\newcommand{\econ}{\end{conj}}
\newcommand{\bcons}{\begin{conjs}}
	\newcommand{\econs}{\end{conjs}}
\newcommand{\bprop}{\begin{prop}}
	\newcommand{\eprop}{\end{prop}}
\newcommand{\br}{\begin{rem}}
	\newcommand{\er}{\end{rem}}
\newcommand{\brs}{\begin{rems}}
	\newcommand{\ers}{\end{rems}}
\newcommand{\bo}{\begin{obser}}
	\newcommand{\eo}{\end{obser}}
\newcommand{\bos}{\begin{obsers}}
	\newcommand{\eos}{\end{obsers}}
\newcommand{\bpf}{\begin{pf}}
	\newcommand{\epf}{\end{pf}}
\newcommand{\ba}{\begin{array}}
	\newcommand{\ea}{\end{array}}
\newcommand{\beq}{\begin{eqnarray}}
	\newcommand{\beqq}{\begin{eqnarray*}}
		\newcommand{\eeq}{\end{eqnarray}}
	\newcommand{\eeqq}{\end{eqnarray*}}

\begin{document}
		\title{Bloch and Landau Type Theorems for Harmonic Mappings with Inhomogeneous Analytic Dilatation}

\author{Vasudevarao Allu}
\address{Vasudevarao Allu,
    Department of Mathematics,
	School of Basic Science,
	Indian Institute of Technology Bhubaneswar,
	Bhubaneswar-752050, Odisha, India.}
\email{avrao@iitbbs.ac.in}
	
	\author{Rohit Kumar}
	\address{Rohit Kumar,
		Department of Mathematics,
		Indraprastha Institute of Information Technology, Delhi,
		New Delhi-110020, India.}
	\email{rohitk12798@gmail.com}

	\subjclass[{AMS} Subject Classification:]{Primary 30C62, 30C25, 31A05}
	
	\maketitle
	
	\begin{abstract}
		We study Bloch and Landau type theorems for a class of sense-preserving harmonic mappings
		$f=h+\overline{g}$ in the unit disk $\mathbb{D}$ satisfying the inhomogeneous analytic dilatation equation
		\[
		g'(z)=\omega(z)h'(z)+\psi(z),
		\]
		where $\omega$ and $\psi$ are analytic functions in $\mathbb{D}$ with $\|\omega\|_{\infty}\leq k<1$ and
		$\|\psi\|_{\infty}\leq M$. Here $h$ and $g$ are called analytic and co-analytic part of $f$, respectively. We first establish a Bloch type theorem for certain normalized class of harmonic functions under the condition $k+M<1$. Finally, we obtain two versions of the Landau theorem under additional assumptions: one for bounded harmonic mappings and another under the assumption that the analytic part of a harmonic functions has bounded Bloch seminorm. 
	\end{abstract}
	
\section{Introduction and Preliminaries}
The classical theorems of Bloch and Landau are among the most fundamental results in geometric function theory. They describe how the normalization of a holomorphic function controls its local injectivity and the size of its image. Let $\mathbb{D} = \{z \in \mathbb{C} : \vert{}z\vert{} < 1\}$ denote the unit disk in the complex plane $\mathbb{C}$.  A disk $D$ is said to be a schlicht disk if there exists a region $\Omega$ in the unit disk $\mathbb{D}$ on which $f$ is univalent on $\Omega$ and $f(\Omega)=D$.   The following foundational result was given by Bloch.
     
     \vspace{2mm}
     \begin{Thm}[\textbf{Bloch Theorem}]\cite{Bloch-1925}
     	Let $f$ be a holomorphic function on $\overline{\mathbb{D}}=\{z:|z|\leq 1\}$ with $|f'(0)|=1$. Then there exists a positive constant $b$ such that $f(\mathbb{D})$ contains a schlicht disk of radius $b$.
     \end{Thm}
    
     The classical Bloch theorem asserts that if $f$ is holomorphic in $\mathbb{D}$ with $f'(0)=1$,
     then there exists a universal constant $\beta>0$, independent of the function $f$, such that the image $f(\mathbb{D})$ contains a schlicht disk of radius at least $\beta$ (see \cite{Bloch-1925,Landau-1929}). The largest possible value of such a constant is called the \emph{Bloch constant}. Although many estimates have been obtained over the last century, the exact value of the Bloch constant is still unknown. Important contributions towards estimating this constant were made by Ahlfors and Grunsky \cite{Ahlfors-Grunsky-1937}, Bonk \cite{Bonk-1990}, and Chen and Gauthier \cite{Chen-Gauthier-1996}.

     \vspace{2mm}
     One of the initial estimates for the Bloch constant is based on related results concerning the univalent schlicht disks of bounded holomorphic mappings. In 1926, Landau \cite{Landau-1926} proved the following result, which calculated the precise value of the radius of this disk. 
     \begin{Thm}\cite{Landau-1926}
     	If $f$ is a holomorphic mapping with $f(0)=0$, $f'(0)=1$ and $|f(z)|<M$ for $z \in \mathbb{D}$, then $f$ is univalent in $\mathbb{D}_{r_0}$, and $f(\mathbb{D}_{r_0})$ contains a disk $\mathbb{D}_{\sigma_0}$, where $$r_0=\frac{1}{M+\sqrt{M^2-1}}\ \ \text{and} \ \ \sigma_0= Mr_0^2.$$
     	The quantities $r_0$ and $\sigma_0$ cannot be improved. The extremal function is $f_0(z)= Mz \left(\frac{1-Mz}{M-z}\right)$.
     \end{Thm}
	
	 The Bloch and Landau theorems have inspired extensive research in geometric function theory. Various extensions have been obtained for several classes of mappings, including harmonic mappings, quasiregular mappings, biharmonic mappings, pluriharmonic mappings, logharmonic mappings, polyharmonic mappings, and $\alpha$-harmonic mappings; see, for example, \cite{Abdulhadi-Muhanna-2008, Chen-Gauthier-Hengartner-2000,Grigoryan-2006,Huang-2008,Liu-Ponnusamy-2024, Liu-Chen-2018,Liu-2008,Zhu-2015}.
	
	\vspace{1mm}
	 During the last four decades, harmonic mappings have become an active area of research because of their close connections with complex analysis, minimal surfaces, quasiconformal mappings, and nonlinear elliptic partial differential equations. A complex-valued harmonic mapping in $\mathbb{D}$ is a function $f = u + iv$, where both $u$ and $v$ are real-valued harmonic functions. We denote the standard complex differential operators by
	$$\frac{\partial}{\partial z} = \frac{1}{2}\left(\frac{\partial}{\partial x} - i\frac{\partial}{\partial y}\right) \quad \text{and} \quad \frac{\partial}{\partial \overline{z}} = \frac{1}{2}\left(\frac{\partial}{\partial x} + i\frac{\partial}{\partial y}\right)$$
	for $z = x + iy \in \mathbb{C}$, where $x$ and $y$ are real. For a continuously differentiable function $f$, we define the quantities:
	$$\Lambda_f = \max_{0 \le \theta \le 2\pi} \vert{}f_z + e^{-2i\theta}f_{\overline{z}}\vert{} = \vert{}f_z\vert{} + \vert{}f_{\overline{z}}\vert{}$$$$\lambda_f = \min_{0 \le \theta \le 2\pi} \vert{}f_z + e^{-2i\theta}f_{\overline{z}}\vert{} = \big\vert{} \vert{}f_z\vert{} - \vert{}f_{\overline{z}}\vert{} \big\vert{}$$
	where $f_z = \partial f / \partial z$ and $f_{\overline{z}} = \partial f / \partial \overline{z}$.   Every harmonic mapping in a simply connected domain admits the canonical decomposition
	\[
	f=h+\overline{g},
	\]
	where $h$ and $g$ are analytic in $\mathbb{D}$. The mapping $f$ is locally univalent if, and only if, its Jacobian
	\[
	J_f=|h'|^2-|g'|^2
	\]
	does not vanish. Furthermore, $f$ is sense-preserving precisely if
	\[
	|g'(z)|<|h'(z)|,\qquad z\in \mathbb{D}.
	\] 
	
	The modern theory of planar harmonic mappings was initiated by the pioneering work of Clunie and Sheil-Small \cite{Clunie-SheilSmall-1984}. Since then, harmonic mappings and their geometric subclasses have been investigated extensively by many researchers.
	
		\vspace{2mm}
	The extension of Bloch and Landau type theorems from holomorphic mappings to harmonic mappings is considerably more difficult because harmonic mappings are generally not conformal and involve both analytic and co-analytic parts. One of the first systematic studies in this direction was carried out by Chen, Gauthier and Hengartner \cite{Chen-Gauthier-Hengartner-2000}, who established two versions of the Landau--Bloch theorem for bounded harmonic mappings. Their estimates were later improved by Grigoryan \cite{Grigoryan-2006}, Huang \cite{Huang-2008}, Liu and Chen \cite{Liu-Chen-2018}, Zhu \cite{Zhu-2015}, and several other authors. 
	
	\vspace{2mm}
	The study of Landau and Bloch type theorems for harmonic mappings has gradually been extended to many subclasses by imposing different geometric or analytic conditions on the mappings. Among the most common assumptions is the bounded analytic dilatation condition
	\[
	g'(z)=\omega(z)h'(z),
	\]
	where $\omega$ is analytic in $\mathbb{D}$ and satisfies
	\[
	\|\omega\|_{\infty}\leq k<1.
	\]
	This condition guarantees that the mapping is quasiregular and provides a direct relation between the analytic and co-analytic parts. Under this assumption, several versions of the Bloch and Landau theorems have been established. In particular, Liu and Chen \cite{Liu-Chen-2018} obtained sharp Landau--Bloch type theorems for harmonic mappings with bounded analytic dilatation by using geometric methods. Their work considerably improved several earlier estimates in the literature.
	
	Another important direction is the study of harmonic mappings satisfying differential inequalities instead of an explicit analytic dilatation. In this setting, Nirenberg \cite{Nirenberg-1953} introduced elliptic mappings, which were later characterized for harmonic mappings by Chen and Ponnusamy \cite{Chen-Ponnusamy-2020}. A harmonic mapping $f=h+\overline{g} $ is $(K,K')$-elliptic if there exist constants $K\geq1$ and $K'\geq0$ such that
	\[
	\Lambda_f^2\leq KJ_f+K',
	\]
	where $\Lambda_f=|h'|+|g'|$ and $J_f$ denotes the Jacobian of $f$. When $K'=0$, this class reduces to the well-known class of $K$-quasiregular harmonic mappings. These classes contain many important families of harmonic mappings and provide a natural framework for extending classical results from conformal mappings to more general settings.
	
	In this paper, we consider the inhomogeneous relation
	$$g'(z) = \omega(z)h'(z) + \psi(z)$$
	where both $\omega$ and $\psi$ are analytic functions in $\mathbb{D}$, with $\Vert{}\omega\Vert{}_\infty \le k < 1$ and $\Vert{}\psi\Vert{}_\infty \le M$. This additional term destroys the homogeneous structure of the classical analytic dilatation equation, meaning standard techniques are no longer directly applicable and the mapping is generally no longer quasiregular in the classical sense, even though it may remain sense-preserving. 
	
	\vspace{1mm}
		\vspace{2mm}
	Let $0 \le k < 1$ and $M \ge 0$. We say that a harmonic mapping $f = h + \overline{g}$ belongs to the class $\mathcal{H}_{k,M}$ if:
	\begin{enumerate}
		\item $h$ and $g$ are analytic in $\mathbb{D}$;
		\item $g(0) = 0$;
		\item there exist analytic functions $\omega, \psi$ in $\mathbb{D}$ such that $g'(z) = \omega(z)h'(z) + \psi(z)$;
		\item $\|\omega\|_\infty \le k$ and $\|\psi\|_\infty \le M$;
		\item $f$ is sense-preserving, \textit{i.e.}, $|g'(z)| < |h'(z)|$ for all $z \in \mathbb{D}$.
	\end{enumerate}
	Observe that when $\psi \equiv 0$, the class $\mathcal{H}_{k,M}$ reduces to the familiar class of harmonic mappings with bounded analytic dilatation. The normalized subclass $\mathcal{H}_{k,M}^0$ consists of those $f \in \mathcal{H}_{k,M}$ satisfying $f(0) = 0$, $h'(0) = 1$, and $g(0) = 0$. At the origin, $|g'(0)| = |\omega(0)h'(0) + \psi(0)| \le k + M$. We specifically impose the structural condition $k + M < 1$, which guarantees that the mapping remains strongly sense-preserving at the normalization scale, ensuring that the Jacobian $J_f(0) = 1 - |g'(0)|^2 > 0$.

	\vspace{2mm}
	Our first main result establishes a Bloch type theorem for the normalized class
	$H_{k,M}^{0}$. We prove that if $k+M<1$, then every mapping in this class maps the
	unit disk $\mathbb{D}$ onto a domain containing a schlicht disk whose radius depends only on
	$k$ and $M$. Then we prove two versions of Landau type theorem for this class of mappings. First, we assume that the harmonic mapping is bounded in the unit disk. Secondly, we assume that the
	analytic part has bounded Bloch semi-norm. 
	
	\vspace{1mm}
	Our results extend the classical Landau--Bloch theory from harmonic mappings
	with homogeneous analytic dilatation to a more general class of harmonic
	mappings with inhomogeneous analytic dilatation. Our result also complement several earlier results obtained in
	\cite{Allu-Kumar-2024, Allu-Kumar-SCV, Allu-Kumar-Alpha,  Chen-Gauthier-Hengartner-2000, Grigoryan-2006, Huang-2008, Liu-Chen-2018, Zhu-2015}.
	
	\vspace{2mm}
	The paper is organized as follows. In Section 2, we state and establish several preliminary results that will be used throughout the paper. In Section 3, we prove the Bloch theorem for the class
	$H_{k,M}^{0}$, and  establish two versions of the Landau theorem under suitable additional assumptions.
	
	\vspace{2mm}
		\section{Auxiliary Results}
	In this section we establish two elementary lemmas that form the main tools in the proofs of our Bloch and Landau theorems. 
	
	\begin{lem}\label{Lemma-1}
		Let $F=H+\overline{G}$ be harmonic in $\mathbb{D}_R$. Suppose $\lambda_F(0) \ge \lambda > 0$ and  for $|z| \le R$, $|H''(z)| + |G''(z)| \le L$. Then for every $0 < s \le R$,
		\[ |F(z_1) - F(z_2)| \ge (\lambda - Ls)|z_1 - z_2| \]
		for all $z_1,z_2 \in \mathbb{D}_s$.
		
		In particular, if $s < \lambda/L$, then $F$ is univalent in $\mathbb{D}_s$. Moreover, if $s \le \lambda/(2L)$, then
		\[ |F(z_1) - F(z_2)| \ge \frac{\lambda}{2}|z_1 - z_2|. \]
	\end{lem}
	\begin{proof}
	Let $z_1,z_2\in\mathbb D_s$ and define the line segment
	\[
	\gamma(t)=z_2+t(z_1-z_2), \qquad 0\le t\le1.
	\]
	Since $\mathbb D_s$ is convex, $\gamma([0,1])\subset\mathbb D_s$. Using the chain rule,
	\[
	\frac{d}{dt}F(\gamma(t))
	=H'(\gamma(t))(\gamma'(t))
	+\overline{G'(\gamma(t))(\gamma'(t))},
	\]
	where $\gamma'(t)=z_1-z_2$. Hence
	\[
	F(z_1)-F(z_2)
	=\int_0^1\left[
	H'(\gamma(t))(z_1-z_2)
	+\overline{G'(\gamma(t))(z_1-z_2)}
	\right]dt.
	\]
	Adding and subtracting the values at the origin, we have
	\begin{align*}
		F(z_1)-F(z_2)
		&=\Big(H'(0)(z_1-z_2)
		+\overline{G'(0)(z_1-z_2)}\Big)  \\
		&\quad
		+\int_0^1\Big[(H'(\gamma(t))-H'(0))(z_1-z_2) +\overline{(G'(\gamma(t))-G'(0))(z_1-z_2)}\Big]dt.
	\end{align*}
	Since
	\[
	\lambda_F(0)
	=\min_{|\xi|=1}
	\big|H'(0)\xi+\overline{G'(0)\xi}\big|,
	\]
	we have
	\[
	\big|H'(0)(z_1-z_2)
	+\overline{G'(0)(z_1-z_2)}\big|
	\ge
	\lambda_F(0)|z_1-z_2|.
	\]
	Next, for each $t\in[0,1]$, the fundamental theorem of calculus gives
	\[
	H'(\gamma(t))-H'(0)
	=\int_0^1
	H''(u\gamma(t))\,\gamma(t)\,du,
	\]
	and therefore
	\begin{equation}\label{Eq-1}
	|H'(\gamma(t))-H'(0)|
	\le |\gamma(t)|
	\int_0^1|H''(u\gamma(t))|\,du.
	\end{equation}
	Similarly,
	\begin{equation}\label{Eq-2}
	|G'(\gamma(t))-G'(0)|
	\le |\gamma(t)|
	\int_0^1|G''(u\gamma(t))|\,du.
	\end{equation}
	By \eqref{Eq-1}, \eqref{Eq-2} and the hypothesis $|H''(z)|+|G''(z)|\le L,$ we obtain
	\[
	|H'(\gamma(t))-H'(0)|
	+
	|G'(\gamma(t))-G'(0)|
	\le Ls.
	\]
	Consequently,
	\begin{align*}
		&
		\left|
		(H'(\gamma(t))-H'(0))(z_1-z_2)
		+\overline{(G'(\gamma(t))-G'(0))(z_1-z_2)}
		\right|
		\\
		&\qquad\le
		\Big(
		|H'(\gamma(t))-H'(0)|
		+
		|G'(\gamma(t))-G'(0)|
		\Big)
		|z_1-z_2|
		\\
		&\qquad\le Ls\,|z_1-z_2|.
	\end{align*}
	Applying the reverse triangle inequality, we conclude that
	\begin{align*}
		|F(z_1)-F(z_2)|
		&\ge
		\lambda_F(0)|z_1-z_2|
		-
		Ls|z_1-z_2| \\
		&\ge
		(\lambda-Ls)|z_1-z_2|.
	\end{align*}
	If $s<\lambda/L$, then $\lambda-Ls>0$, and hence
	\[
	|F(z_1)-F(z_2)|>0
	\]
	whenever $z_1\ne z_2$. Thus $F$ is univalent in $\mathbb D_s$.
	\vspace{2mm}
	
	\noindent Finally, if $s\le\lambda/(2L)$, then
	\[
	\lambda-Ls\ge\frac{\lambda}{2},
	\]
	which gives
	\[
	|F(z_1)-F(z_2)|
	\ge
	\frac{\lambda}{2}|z_1-z_2|.
	\]
	This completes the proof.
\end{proof}
	\begin{lem}\label{Lemma-2}
		Let $F$ be harmonic and univalent on $\overline{\mathbb{D}}_s$, and suppose $|F(z) - F(0)| \ge R$ for $|z| = s$. Then $\mathbb{D}(F(0), R) \subset F(\mathbb{D}_s)$. Consequently, if $|F(z) - F(0)| \ge \rho$ on $\partial\mathbb{D}_s$, then $F(\mathbb{D})$ contains a schlicht disk of radius $\rho$.
	\end{lem}
	
	\begin{proof}
		Since $F$ is continuous and injective on $\overline{\mathbb{D}}_s$, the curve $F(\partial\mathbb{D}_s)$ is a Jordan curve. The image $F(\mathbb{D}_s)$ is the bounded component enclosed by this curve. If the boundary curve stays outside the Euclidean disk $\mathbb{D}(F(0), R)$, then the whole disk $\mathbb{D}(F(0), R)$ lies in the bounded component.
	\end{proof}
	\begin{Lem}\label{Lemma-A}\cite[Theorem 3]{Colonna-1989}
	Let $F$ be a complex-valued harmonic function in $\mathbb{D}$, and suppose
	$$|F(z)| \le C, \quad z \in \mathbb{D}.$$
	Then $$\Lambda_F(z) := |F_z(z)| + |F_{\overline{z}}(z)| \le \frac{4C}{\pi} \frac{1}{1-|z|^2}.$$Equivalently, if $F = H + \bar{G}$, then:$$|H'(z)| + |G'(z)| \le \frac{4C}{\pi} \frac{1}{1-|z|^2}.$$
	\end{Lem}
The above lemmas provide the quantitative estimates required in the proofs of our main results.

	\section{Bloch and Landau Type Results for the Class $H_{k,M}$}
	In this section, we establish the main results of the paper. We first prove a Bloch type theorem for the normalized class $H_{k,M}^{0}$. We then show that a universal Landau theorem does not hold for this class by constructing an explicit counterexample. Finally, we obtain two versions of the Landau theorem under natural additional assumptions, namely boundedness of the harmonic mapping and boundedness of the Bloch seminorm of its analytic part. We begin with the following Bloch type theorem for the normalized class $H_{k,M}^{0}$.
	
	\begin{thm}
		Let $f = h + \overline{g} \in \mathcal{H}_{k,M}^0$ and assume $k + M < 1$. Set $q := k + M$ and $\lambda := 1 - q > 0$. Then $f(\mathbb{D})$ contains a schlicht disk of radius at least
		\[ B_{k,M} = \frac{3(1 - k - M)^2}{128}. \]
	\end{thm}
	
	\begin{proof}
		Fix $0 < r < 1$. Define $$W_r(z) = (r^2 - |z|^2)|h'(z)|$$ for $|z| \le r$. Since $W_r$ is continuous and vanishes on $\partial\mathbb{D}_r$, it attains its maximum at some point $z_0 \in \mathbb{D}_r$. Set $C := W_r(z_0)$. Since $h'(0) = 1$, $C \ge W_r(0) = r^2$. Thus $C \ge r^2$. Let
		\[ T(\zeta) = \frac{r(z_0 + r\zeta)}{r + \overline{z_0}\zeta}. \]
		Then $T$ maps $\mathbb{D}$ conformally onto $\mathbb{D}_r$ and $T(0) = z_0$. Moreover,
		\[ |T'(\zeta)| = \frac{r(r^2 - |z_0|^2)}{|r + \overline{z_0}\zeta|^2}.\]
		A simple computation shows that
		\[ (1 - |\zeta|^2)|T'(\zeta)| = \frac{r^2 - |T(\zeta)|^2}{r}. \]
		Define the harmonic mapping
		\[ F(\zeta) = H(\zeta) + \overline{G(\zeta)} := \frac{r}{C} (f(T(\zeta)) - f(z_0)). \]
		Then $H'(\zeta) = \frac{r}{C} h'(T(\zeta)) T'(\zeta)$, and therefore
		\[ |H'(0)| = \frac{r}{C} |h'(z_0)| |T'(0)|. \]
		Since $|T'(0)| = (r^2 - |z_0|^2)/r$ and $C = (r^2 - |z_0|^2)|h'(z_0)|$, we have $|H'(0)| = 1$.
		Furthermore,
		\[ (1 - |\zeta|^2)|H'(\zeta)| = \frac{W_r(T(\zeta))}{C} \le 1. \]
		Hence $|H'(\zeta)| \le \frac{1}{1 - |\zeta|^2}$. In particular, for $|\zeta| \le 1/2$, $|H'(\zeta)| \le 4/3$.
		
		\vspace{2mm}
		Since $f$ is sense-preserving, the map $F$ is also sense-preserving. Therefore $|G'(\zeta)| < |H'(\zeta)|$, which implies $|G'(\zeta)| \le 4/3$ for $|\zeta| \le 1/2$.
		
		By Cauchy's estimate applied to the analytic functions $H'$ and $G'$, for $|\zeta| \le 1/4$, we have $|H''(\zeta)| \le 16/3$ and $|G''(\zeta)| \le 16/3$. Thus, we have
		\[  |H''(\zeta)| + |G''(\zeta)| \le \frac{32}{3}, \quad |\zeta| \le \frac{1}{4}. \]
		
		We next estimate the lower norm of $D_F(0)$. The equation $g' = \omega h' + \psi$ transforms into $G'(\zeta) = \widetilde{\omega}(\zeta)H'(\zeta) + \widetilde{\psi}(\zeta)$, where $\widetilde{\omega}(\zeta) = \omega(T(\zeta))$ and $\widetilde{\psi}(\zeta) = \frac{r}{C}\psi(T(\zeta))T'(\zeta)$. At $\zeta = 0$,
		\[ |\widetilde{\psi}(0)| \le \frac{r}{C} M \frac{r^2 - |z_0|^2}{r} = \frac{M(r^2 - |z_0|^2)}{C}. \]
		Since $C = (r^2 - |z_0|^2)|h'(z_0)|$, we get $$|\widetilde{\psi}(0)| \le \frac{M}{|h'(z_0)|}.$$
		Note that
		\[ |h'(z_0)| = \frac{C}{r^2 - |z_0|^2} \ge \frac{r^2}{r^2 - |z_0|^2} \ge 1. \]
		Therefore $|\widetilde{\psi}(0)| \le M$. Also, $|\widetilde{\omega}(0)| \le k$. Since $|H'(0)| = 1$, $|G'(0)| \le k + M = q$. Hence
		\[ \lambda_f(0) = |H'(0)| - |G'(0)| \ge 1 - q = \lambda. \]
		Applying Lemma \ref{Lemma-1} with $L = 32/3$ and $ s = \frac{\lambda}{2L} = \frac{3\lambda}{64}$, we conclude that  $F$ is univalent on $\mathbb{D}_s$, and 
		$$|F(\zeta) - F(0)| \ge \frac{\lambda}{2}|\zeta|$$ for $|\zeta| \le s$. In particular, on $|\zeta| = s$,
		\[ |F(\zeta) - F(0)| \ge \frac{\lambda s}{2} = \frac{3\lambda^2}{128}. \]
		By Lemma \ref{Lemma-2}, $F(\mathbb{D}_s)$ contains a schlicht disk of radius $\rho = \frac{3\lambda^2}{128}$. Returning to the original map, $f(T(\mathbb{D}_s))$ contains a schlicht disk of radius $\frac{C}{r} \rho$. Since $C \ge r^2$, we have $\frac{C}{r} \ge r$. Thus, $f(\mathbb{D})$ contains a schlicht disk of radius at least $r\rho$. Since this holds for any $r \in (0, 1)$, we can take the supremum as $r \to 1$ to obtain the uniform lower bound:
		\[ B_{k,M} = \rho = \frac{3\lambda^2}{128} = \frac{3(1 - k - M)^2}{128}. \]
		This completes the proof of the theorem.
	\end{proof}

	The above Bloch theorem is valid, but the corresponding Landau theorem is false without additional  boundedness assumptions. A Landau theorem would assert the existence of a radius $r_0 = r_0(k,M) > 0$ such that every $f \in \mathcal{H}_{k,M}^0$ is univalent in $\mathbb{D}_{r_0}$. The following example shows that no such result can hold.
	
	\begin{thm}\label{Theroem-2}
		There is no positive constant $r_0 = r_0(k,M)$ depending only on $k$ and $M$ such that every mapping in $\mathcal{H}_{k,M}^0$ is univalent in $\mathbb{D}_{r_0}$. In fact, the assertion fails already when $k = M = 0$.
	\end{thm}
	
	\begin{proof}
		Consider the function $f_N(z) = \frac{e^{Nz} - 1}{N}$ for $N \in \mathbb{N}$. The function $f_N$ is analytic, which implies $g \equiv 0$. Consequently, $f_N$ belongs to the class $\mathcal{H}_{k,M}$ characterized by $\omega \equiv 0$ and $\psi \equiv 0$, that is, $k = M = 0$.
		
		Moreover, $f_N(0) = 0$ and $f'_N(0) = 1$, so $f_N \in \mathcal{H}_{0,0}^0$. Also, $f'_N(z) = e^{Nz} \neq 0$, so $f_N$ is locally univalent and sense-preserving.
		
		However,
		\[ f_N\left(i\frac{\pi}{N}\right) = \frac{e^{i\pi} - 1}{N} = -\frac{2}{N}, \quad \text{and} \quad f_N\left(-i\frac{\pi}{N}\right) = \frac{e^{-i\pi} - 1}{N} = -\frac{2}{N}. \]
		Thus $f_N$ is not injective on any disk $\mathbb{D}_r$ with $r > \frac{\pi}{N}$. Given any proposed $r_0 > 0$, choose $N > \frac{\pi}{r_0}$. Then $f_N$ is not univalent in $\mathbb{D}_{r_0}$. Hence no universal Landau radius depending only on $k,M$ exists.
	\end{proof}
	Although Theorem \ref{Theroem-2} shows that no universal Landau theorem can hold for the entire class $\mathcal{H}_{k,M}^0$ , we can estbalish Landau type theorem if suitable global bounds are imposed on the mapping. We first consider bounded harmonic mappings and then the harmonic mappings whose analytic part has bounded Bloch seminorm.
	\begin{thm}\label{Theorem-3}
		Let $f = h + \overline{g} \in \mathcal{H}_{k,M}^0$, and suppose for some $C>0,$ $|f(z)| \le C$ for $ z \in \mathbb{D}.$ If $k+M < 1$, then $f$ is univalent in the disk $\mathbb{D}_{r_C}$ with 
		$$r_C = \frac{3\pi(1-k-M)}{256C}.$$
		Moreover, $f(\mathbb{D}_{r_C})$ contains the disk centered at $0$ of radius $$R_C = \frac{3\pi(1-k-M)^2}{512C}.$$That is,$\{w \in \mathbb{C} : |w| < R_C\} \subset f(\mathbb{D}_{r_C}).$
	\end{thm}
	\begin{proof}
		Since $|f(z)| \le C$, by Lemma \ref{Lemma-A}, we have 
		$$|h'(z)| + |g'(z)| \le \frac{4C}{\pi} \frac{1}{1-|z|^2}.$$
		In particular, for $|z| \le 1/2$, we have
		$$|h'(z)| + |g'(z)| \le \frac{4C}{\pi} \cdot \frac{1}{1-1/4} = \frac{16C}{3\pi}.$$Hence, we can say that
		 $$|h'(z)| \le \frac{16C}{3\pi}, \quad |g'(z)| \le \frac{16C}{3\pi}, \quad |z| \le \frac{1}{2}.$$
		Now take a point $z$ such that $|z| \le 1/4$. The disk $D(z, 1/4)$ lies entirely inside $\mathbb{D}_{1/2}$. Applying Cauchy's estimate to the analytic functions $h'$ and $g'$ over this disk of radius $1/4$, we obtain
		$$|h''(z)| \le 4 \cdot \frac{16C}{3\pi} = \frac{64C}{3\pi},$$
		and similarly,
		$$|g''(z)| \le \frac{64C}{3\pi}.$$
		Therefore, 
		$$ |h''(z)| + |g''(z)| \le \frac{128C}{3\pi}, \quad |z| \le \frac{1}{4}.$$
		Since $h'(0) = 1$, the equation $g'(z) = \omega(z)h'(z) + \psi(z)$ evaluated at $z = 0$ gives 
		$$g'(0) = \omega(0)h'(0) + \psi(0) = \omega(0) + \psi(0).$$
		Because $\|\omega\|_\infty \le k$ and $\|\psi\|_\infty \le M$, this implies that 
		$$|g'(0)| \le k + M.$$
		Thus, setting $\lambda := 1 - k - M > 0$, we have 
		$$\lambda_f(0) = |h'(0)| - |g'(0)| \ge 1 - k - M = \lambda.$$
		Now, applying Lemma \ref{Lemma-1} for  $L= \frac{128C}{3\pi}$ and$$r_C := \frac{\lambda}{2L} = \frac{1-k-M}{2} \cdot \frac{3\pi}{128C} = \frac{3\pi(1-k-M)}{256C},$$ 
		for any $z_1, z_2 \in \mathbb{D}_{r_C}$, we have  
		$$|f(z_1) - f(z_2)| \ge (\lambda - Lr_C)|z_1 - z_2| = \left(\lambda - \frac{\lambda}{2}\right)|z_1 - z_2| = \frac{\lambda}{2}|z_1 - z_2|.$$
		Because this lower bound is strictly positive for $z_1 \neq z_2$, $f$ is univalent on $\mathbb{D}_{r_C}$.
		
		\vspace{1mm}
		Next, by setting $z_2 = 0$ and using $f(0) = 0$, we get
		$$|f(z)| \ge \frac{\lambda}{2}|z|, \quad |z| \le r_C.$$Thus, on the boundary $|z| = r_C$, we have 
		$$|f(z)| \ge \frac{\lambda}{2} r_C = \frac{\lambda}{2} \cdot \frac{3\pi\lambda}{256C} = \frac{3\pi\lambda^2}{512C}.$$Since $f$ is univalent on $\mathbb{D}_{r_C}$ and continuous on the boundary, Lemma \ref{Lemma-2} guarantees that the image $f(\mathbb{D}_{r_C})$ contains the open disk centered at $0$ of radius, where 
		$$R_C = \frac{3\pi\lambda^2}{512C} = \frac{3\pi(1-k-M)^2}{512C}.$$
		Therefore $$\{w \in \mathbb{C} : |w| < R_C\} \subset f(\mathbb{D}_{r_C}).$$This completes the proof. 
	\end{proof}

	The obstruction in Theorem \ref{Theroem-2} also comes from the absence of any global control on $h'$. A natural additional assumption is boundedness of the analytic Bloch seminorm:
	\[ \beta(h) = \sup_{z\in\mathbb{D}} (1 - |z|^2)|h'(z)|. \]
	Under this assumption, one can obtain a new version of Landau theorem.
	
	\begin{thm}
		Let $f = h + \overline{g} \in \mathcal{H}_{k,M}^0$ with $k + M < 1$. Assume additionally that
		\[ \beta(h) = \sup_{z\in\mathbb{D}} (1 - |z|^2)|h'(z)| \le A \]
		for some $A \ge 1$. Set $\lambda = 1 - k - M$. Then $f$ is univalent in the disk
		\[ \mathbb{D}_{r_A}, \quad r_A = \frac{3(1 - k - M)}{64A}. \]
		Moreover, $f(\mathbb{D}_{r_A})$ contains a disk centered at $0$ of radius at least
		\[ R_{k,M,A} = \frac{3(1 - k - M)^2}{128A}. \]
	\end{thm}
	
	\begin{proof}
		Since $\beta(h) \le A$, we have 
		$$|h'(z)| \le \frac{A}{1 - |z|^2}.$$
		Since $f$ is sense-preserving, we have $|g'(z)| < |h'(z)|$, and hence, $$|g'(z)| \le \frac{A}{1 - |z|^2}.$$
        Thus for $|z|\le 1/2$, we have 
        $$ |h'(z)|\le \frac{4A}{3}, \quad |g'(z)|\le \frac{4A}{3}.$$
        The remaining arguments are identical to those in the proof of Theorem \ref{Theorem-3} and are therefore omitted.
	\end{proof}
	
	\vspace{1.5mm}
	
	\noindent\textbf{Compliance of Ethical Standards:}\\
	\noindent\textbf{Conflict of interest.} The authors declare that there is no conflict  of interest regarding the publication of this paper.
	\vspace{1.5mm}
	
	\noindent\textbf{Data availability statement.}  Data sharing is not applicable to this article as no datasets were generated or analyzed during the current study.\vspace{1.5mm}
	
	\noindent\textbf{Authors contributions.} Both the authors have made equal contributions in reading, writing, and preparing the manuscript.


\begin{thebibliography}{99}
		\bibitem{Abdulhadi-Muhanna-2008}
		{\sc Z. Abdulhadi } and {\sc Y. Abu Muhanna},
		Landau's theorem for biharmonic mappings,
		{\it J. Math. Anal. Appl.}
		{\bf 338}(1) (2008), 705--709.
		
		\bibitem{Ahlfors-Grunsky-1937}
		{\sc L. Ahlfors} and {\sc H. Grunsky},
		Uber die Blochsche Konstante,
		{\it Math. Z.}
		{\bf 42} (1937), 671--673.
		
		\bibitem{Allu-Kumar-2024}
		{\sc V. Allu } and {\sc R. Kumar},
		Landau--Bloch type theorem for elliptic and quasiregular harmonic mappings,
		{\it J. Math. Anal. Appl.}
		{\bf 535} (2024), 128215.
		
		\bibitem{Allu-Kumar-SCV}
		{\sc V. Allu } and {\sc R. Kumar},
		The Landau--Bloch type theorems for certain class of holomorphic and pluriharmonic mappings in $\mathbb{C}^n$, {\it Ann. Fenn. Math.} {\bf 50}(1) (2025), 215--219.
		
		\bibitem{Allu-Kumar-Alpha}
		{\sc V. Allu } and {\sc R. Kumar},
		Landau type theorem for generalized harmonic mappings,
		{\it Complex Var. Elliptic Equ.}
		(2026), DOI:10.1080/17476933.2026.2682927.
		
		\bibitem{Bloch-1925}
		{\sc A. Bloch},
		Les théorèmes de M. Valiron sur les fonctions entières et la théorie de l'uniformisation,
		{\it Ann. Fac. Sci. Univ. Toulouse}
		{\bf 17} (1925), 1--22.
		
		\bibitem{Bonk-1990}
		{\sc M. Bonk},
		On Bloch's constant,
		{\it Proc. Amer. Math. Soc.}
		{\bf 110}(4) (1990), 889--894.
		
		\bibitem{Chen-Gauthier-1996}
		{\sc S. Chen} and {\sc P. Gauthier},
		The lower bound of Bloch's constant,
		{\it Proc. Amer. Math. Soc.}
		{\bf 124}(7) (1996), 2155--2158.
		
		\bibitem{Chen-Gauthier-Hengartner-2000}
		{\sc S. Chen, P. Gauthier} and {\sc W. Hengartner},
		Bloch constants for planar harmonic mappings,
		{\it Proc. Amer. Math. Soc.}
		{\bf 128}(11) (2000), 3231--3240.
		
		\bibitem{Chen-Ponnusamy-2020}
		{\sc S. Chen} and {\sc S. Ponnusamy},
		Characterizations of elliptic mappings,
		{\it Complex Var. Elliptic Equ.}
		{\bf 65}(11) (2020), 1911--1926.
		
		\bibitem{Clunie-SheilSmall-1984}
		{\sc J. Clunie} and {\sc T. Sheil-Small},
		Harmonic univalent functions,
		{\it Ann. Acad. Sci. Fenn. Ser. A I Math.}
		{\bf 9} (1984), 3--25.
		
		\bibitem{Colonna-1989}
		{\sc F. Colonna},
		The Bloch constant of bounded harmonic mappings,
		{\it Indiana Univ. Math. J.}
		{\bf 38}(4) (1989), 829--840.
		
		
		\bibitem{Grigoryan-2006}
		{\sc A. Grigoryan},
		Landau and Bloch theorems for harmonic mappings,
		{\it Complex Var. Elliptic Equ.}
		{\bf 51}(1) (2006), 81--87.
		
		\bibitem{Huang-2008}
		{\sc X. Z. Huang},
		Estimates on Bloch constants for planar harmonic mappings,
		{\it J. Math. Anal. Appl.}
		{\bf 337}(2) (2008), 880--887.
		
		\bibitem{Landau-1926}
		{\sc E. Landau},
		Der Picard-Schottkysche Satz und die Blochsche Konstanten,
		{\it Sitzungsber. Preuss. Akad. Wiss. Berlin Phys.-Math. Kl.}
		(1926), 467--474.
		
		\bibitem{Landau-1929}
		{\sc E. Landau},
		Darstellung und Begründung einiger neuerer Ergebnisse der Funktionentheorie,
		Springer, Berlin, 1929.
		
		\bibitem{Liu-2008}
		{\sc M. S. Liu},
		Landau's theorems for biharmonic mappings,
		{\it Complex Var. Elliptic Equ.}
		{\bf 53}(9) (2008), 843--855.
		
%
		
		\bibitem{Liu-Chen-2018}
		{\sc M. S. Liu} and {\sc H. H. Chen},
		The Landau--Bloch type theorems for planar harmonic mappings with bounded dilation,
		{\it J. Math. Anal. Appl.}
		{\bf 468}(2) (2018), 1066--1081.
		
		\bibitem{Liu-Ponnusamy-2024}
		{\sc M. S. Liu} and {\sc S. Ponnusamy},
		Landau-type theorems for certain bounded bi-analytic functions and biharmonic mappings,
		{\it Canad. Math. Bull.}
		{\bf 67}(1) (2024), 152--165.
		
		   \bibitem{Nirenberg-1953} {\sc L. Nirenberg}, On nonlinear elliptic partial differential equations and H\"{o}lder continuity, {\it  Commun. Pure. Appl. Math.} {\bf 6} (1953), 103--156.
		
		\bibitem{Zhu-2015}
		{\sc J. F. Zhu},
		Landau theorem for planar harmonic mappings,
		{\it Complex Anal. Oper. Theory}
		{\bf 9}(8) (2015), 1819--1826.
	\end{thebibliography}
\end{document}